\documentclass[11pt]{article}
\usepackage[utf8]{inputenc}
\usepackage{amssymb}
\usepackage{amsmath}
\usepackage{amsthm}
\usepackage{siunitx}
\usepackage{comment}
\usepackage{indentfirst}
\usepackage{color,soul}
\usepackage{longtable}

\usepackage[citation-order]{amsrefs}

\def\setupbib{\catcode`@=\active}
\begingroup\lccode`~=`@
  \lowercase{\endgroup\def~}#1#{\gatherkey{#1}}
\def\gatherkey#1#2{\gatherkeyaux{#1}#2\gatherkeyaux}
\def\gatherkeyaux#1#2,#3\gatherkeyaux{\bib{#2}{#1}{#3}}

\makeatletter
\def\bbordermatrix#1{\begingroup \m@th
  \@tempdima 4.75\p@
  \setbox\z@\vbox{%
    \def\cr{\crcr\noalign{\kern2\p@\global\let\cr\endline}}%
    \ialign{$##$\hfil\kern2\p@\kern\@tempdima&\thinspace\hfil$##$\hfil
      &&\quad\hfil$##$\hfil\crcr
      \omit\strut\hfil\crcr\noalign{\kern-\baselineskip}%
      #1\crcr\omit\strut\cr}}%
  \setbox\tw@\vbox{\unvcopy\z@\global\setbox\@ne\lastbox}%
  \setbox\tw@\hbox{\unhbox\@ne\unskip\global\setbox\@ne\lastbox}%
  \setbox\tw@\hbox{$\kern\wd\@ne\kern-\@tempdima\left[\kern-\wd\@ne
    \global\setbox\@ne\vbox{\box\@ne\kern2\p@}%
    \vcenter{\kern-\ht\@ne\unvbox\z@\kern-\baselineskip}\,\right]$}%
  \null\;\vbox{\kern\ht\@ne\box\tw@}\endgroup}
\makeatother
\newtheorem*{corollary*}{Corollary}
\newtheorem*{theorem*}{Theorem}

\theoremstyle{definition}
\newtheorem{theorem}{Theorem}
\newtheorem{lemma}[theorem]{Lemma}
\newtheorem{corollary}[theorem]{Corollary}
\newtheorem{definition}[theorem]{Definition}

\theoremstyle{remark}

\theoremstyle{remark}
\newtheorem{example}{Example}

\theoremstyle{remark}
\newtheorem*{example*}{Example}

\theoremstyle{remark}
\newtheorem*{remark}{Remark}

\theoremstyle{remark}

\def\0{{\bm 0}}   

\title{On a class of optimal constant weight ternary codes }
\author{
 Hadi Kharaghani\thanks{Department of Mathematics and Computer Science, University of Lethbridge,
Lethbridge, Alberta, T1K 3M4, Canada. \texttt{kharaghani@uleth.ca}},
\and
  Sho Suda\thanks{Department of Mathematics,  National Defense Academy of Japan, Yokosuka, Kanagawa 239-8686, Japan. \texttt{ssuda@nda.ac.jp}},
\and
Vlad Zaitsev\thanks{Department of Mathematics and Computer Science, University of Lethbridge,
Lethbridge, Alberta, T1K 3M4, Canada. \texttt{vlad.zaitsev@uleth.ca}}
}
\date{\today}

\begin{document}
\maketitle

\begin{abstract}
   A weighing matrix $W$ of order $n=\frac{p^{m+1}-1}{p-1}$ and weight $p^m$ is constructed and shown that
    the rows of $W$ and  $-W$ together form optimal constant weight ternary codes of length $n$, weight $p^m$ and minimum distance $p^{m-1}(\frac{p+3}{2})$ for each odd prime power $p$ and integer $m\ge 1$ and thus
    $$A_3\left(\frac{p^{m+1}-1}{p-1},p^{m-1}\big(\frac{p+3}{2}\big),p^{m}\right)=2\big(\frac{p^{m+1}-1}{p-1}\big).$$
\end{abstract}

\section{Introduction}
It is not hard to see that the rows of the  incidence matrix of any symmetric design form optimal binary codes. 
 A \emph{weighing matrix} of order $n$ and weight $p$, shown as $W(n,p)$, is a $(0,\pm 1)$-matrix $W$ of order $n$ such that $WW^t = pI_n$. The case where $n=p+1$ is called a \emph{conference} matrix and $n=p$ is a \emph{Hadamard} matrix. Optimal binary codes obtained from Hadamard matrices constitute an important class of codes due to their error-correcting capability. One expects that weighing matrices also provide sets of useful codes.   The rows of a weighing matrix $W(n,p)$, $n\ne p$, form a set of constant weight \emph{ternary} codes. Theorem 16 of \cite{oster} relates, though in disguise, to an optimal class of constant ternary codes from weighing matrices. There seems to be no more optimal constant weight codes explicitly related to the weighing matrices in the literature. A large class of optimal constant weight ternary codes are shown to arise from weighing matrices in Section 3 and Section 4 of this paper.
 
\section{Preliminaries}
Let $S_3=\{0,1,-1\}$. A \emph{ternary} code of length $n$ is any subset $C$ of $S_3^n$. Elements of $C$ are called \emph{codewords}. The \emph{Hamming distance}  between two ternary codewords of length $n$ is the number of coordinates in which they differ. The number of nonzero entries of a codeword is the \emph{weight} of the code. A ternary code of length $n$ containing $M$ codewords and having minimum Hamming distance $d$ is denoted $(n,M,d)$-code. If all the codewords have the same number of nonzero entires the code is said to be of \emph{constant weight}. The largest value of $M$ for which there is a ternary code of length $n$, minimum distance $d$ and weight $w$ is denoted by $A_3(n,d,w)$ and the code is said to be \emph{optimal}. 
$\ddot{O}$sterg$\dot{a}$rd et al in \cite{oster} among other interesting results have shown that if $p\ge 3$ is a prime power and $m\ge 1$, then
 $A_3(p^m + 1,(p^m + 3)/2,p^m)=2(p^m+1)$. To show an extension of this result the restricted Johnson bound for $A_3(n,d,w)$ is essential, see Theorem 2.3.4 of \cite{pless}.

\begin{theorem}\label{J1}
If $3w^2-4nw+2nd>0$, then
\begin{align}
A_3(n,d,w)&\le \bigg\lfloor\frac{2nd}{3w^2-4nw+2nd}\bigg\rfloor. \label{eq:1}
\end{align}
\end{theorem}

A weighing matrix $W(n,p)$ is said to be in \emph{normal form}  if  {it} has the block configuration 
$$
\begin{bmatrix} \mathbf{0}_{n-p} & R \\ \mathbf{1}_p & D  \end{bmatrix},
$$
for some  $(0,\pm 1)$-matrices $R$ 
and $D$, where $\mathbf{0}_{n-p}$ is the  $(n-p)\times 1$ column vector with all entries $0$ and $\mathbf{1}_p$ is the $p\times 1$ column vector of all entries $1$.  We call $R$ the \emph{residual} and $D$  the \emph{derived} parts of the weighing matrix. It follows  that $RR^t = pI_{n-p}$, and $DD^t = pI_{p} - J_p$, and $RD^t = DR^t = 0$.  {By} permuting and negating some rows, if necessary, every weighing matrix can be assumed to be in normal form. The Jacobsthal matrix, as described below, is used extensively in this paper, see \cite{see-yam}.

\begin{theorem}\label{Jacob}
There is a $(0,\pm1)$-matrix $Q$ of  an odd prime power order $p$ having zero on the diagonal and $\pm 1$ off-diagonal with row and column sum zero and $QQ^t=pI_p - J_p$.
\end{theorem}

\begin{theorem}\label{con-code}
Let $C$ be a conference matrix $W(n+1,n)$ with the $n\times n$ matrix $D$ being the derived part of $C$. The rows of $D$ form an optimal constant weight ternary code with minimum distance $\frac{n+3}{2}$ and $A_3(n,\frac{n+3}{2},n-1)=n$.
\end{theorem}
\begin{proof}
The inner product of two distinct rows of $D$ is $-1$. There are thus $(n-1)/2$ minus ones and $(n-3)/2$ plus ones in the inner product. It follows that the distance between any two rows is $\frac{n-1}{2}+2=\frac{n+3}{2}$. Considering that $3w^2-4nw+2nd=3(n-1)^2-4n(n-1)+2n(\frac{n+3}{2})=n+3$, it follows from Johnson bound (\ref{eq:1}) above that $A_3(n,\frac{n+3}{2},n-1)\le n$. This completes the proof.
\end{proof}

A second Johnson bound which will be used is as follows.

\begin{theorem}\label{J2} 
\begin{align}
A_3(n,d,w)&\le \bigg\lfloor\frac{2n}{w}A_3(n-1,d,w-1).\bigg\rfloor \label{eq:2}
\end{align}
\end{theorem}
 
 As an application of Johnson bound (\ref{eq:2}) and Theorem \ref{con-code} above a large set of optimal ternary codes are obtained in the next theorem.
 
 \begin{theorem}\label{con-main}
 Let $C$ be a conference matrix of order $n+1$. Then the rows of $C$ and $-C$ together form an optimal constant weight ternary code and so $$A_3(n+1,\frac{n+3}{2},n)=2(n+1).$$
 \end{theorem}
 \begin{proof}
 From Theorem \ref{con-code} we know that $A_3(n,\frac{n+3}{2},n-1)= n$.  By Johnson bound (\ref{eq:2}):
 \begin{align*}A_3(n+1,\frac{n+3}{2},n)&\le \frac{2(n+1)}{n}A_3(n,\frac{n+3}{2},n-1)\\&=\frac{2(n+1)}{n}n=2(n+1).\end{align*}
 For convenience we may assume that 
 \[
 \begin{bmatrix} C\\-C\end{bmatrix}=\begin{bmatrix} 0 & \mathbf{1}_n^t\\\mathbf{1}_n & D\\0 &-\mathbf{1}_n^t \\-\mathbf{1}_n & -D \end{bmatrix}.\] 
 The codewords are of length $n+1$ and weight $n$. The minimum distance in $C$ and $-C$ stays as $\frac{n+3}{2}$. The only part requiring attention is taking a row $i$ of $D$ and $-j$ in $-D$. If $i=|-j|$, then the distance between  row $i$ of $D$ and $-j$ of $-D$ is $n-1$ and thus the distance between the two longer rows is $n$. For $i\ne |-j|$ the inner product of row $i$ of $D$ and $-j$ of $-D$ is one and the distance between the two rows is $\frac{n-3}{2}+2=\frac{n+1}{2}$. The entries in the first column now make the difference and the minimum distance between the longer rows stays as $\frac{n+3}{2}$. This completes the proof.
 \end{proof}
 Theorem \ref{con-main} extends Theorem 16 of \cite{oster}. 
  
 \begin{corollary}\label{cor}
 Let $p$ be an odd prime power, then $A_3(p^m+1,\frac{p^m+3}{2},p^m)=2(p^m+1)$ for every positive integer $m$.
 \end{corollary}
 \begin{proof}
 There is a conference matrix $W(p^m+1,p^m)$ for every odd prime power $p$ and positive integer $m$ and 
the result follows.
\end{proof}
\begin{remark} There are many orders not covered by the Theorem 16 of \cite{oster}. For example, there is a conference matrix $W(16,15)$ and $15$ is not a prime number.
\end{remark}

The main result of the paper is in part an application  of orthogonal arrays. 
\begin{definition}
Let $S = \{1,2,\dots,q\}$ be some finite alphabet. An \emph{orthogonal array} of strength $t$ and index $\lambda$ is an $N \times k$ matrix over $S$ such that in every $N \times t$ subarray, each $t$-tuple in $S^t$ appears $\lambda$ times. We denote this property as $OA_\lambda(N,k,q,t)$. 
\end{definition}
\begin{theorem}\label{oa}
For the prime power $p$ and the positive integer $m$ there is a $$p^{m+1}\times\left(\frac{p^{m+1}-1}{p-1}\right)$$ 
array ${\bf O}$ in  $p$ symbols  
such that any two distinct rows share a common symbol in exactly $$\frac{p^m-1}{p-1}$$  columns.
\end{theorem}

Theorem \ref{oa}, see \cite{KPS} for details, will be used extensively in the next section. The main reference for orthogonal arrays is  \cite{Hed}.

\section{The main construction}

There was no condition imposed on the weight of a conference matrix in Theorem \ref{con-main}. In this section the flat assumption is that all conference matrices have odd prime power weights, unless otherwise specified.  The following Lemma is immediate.

\begin{lemma}
Let $W$ be a weighing matrix $W(n,p)$ whose rows form constant  ternary codes of length $n$, weight $p$ and minimum distance $d$. Then the rows of 
$W\otimes \begin{bmatrix} \mathbf{1}  ^t_q\end{bmatrix}$ are mutually orthogonal and form constant weight ternary codes of length $qn$, weight $pq$ and minimum distance $qd$.
\end{lemma} 
 
 The main construction is recursive and consists of two parts. Before proceeding to the main construction an example is helpful.
 
 \begin{example}
 For the prime $p=3$ 
 \begin{align}  A_3(13,9,9)= & \,26. \label{eq:3}
 \end{align}
 In order to show equation (\ref{eq:3}) the construction is broken into seven steps.
 \begin{enumerate}
 \item Starting with a normalized $W(4,3)$ ($-1$ is denoted by $-$):
 \[W=\begin{bmatrix} 0 & 1 & 1 &1\\1 & 0 & 1 &-\\1 & - & 0 & 1\\1 & 1 & - &0 \end{bmatrix}.\]
 \item The derived part of $W$ is
  \[D=\begin{bmatrix} 0 & 1 &-\\- & 0 & 1\\1 & - &0 \end{bmatrix}.\]
\item A corresponding orthogonal array from Theorem \ref{oa} on symbols $\{1,2,3\}$ is
\small{
\[\begin{array}{cccc}
1 & 1 & 1 &1\\
1 & 2 & 2 & 2\\
1 & 3 & 3 &3\\
2 & 1 & 2 &3\\
2 & 3 & 1 & 2\\
2 & 2 & 3 &1\\
3 & 1 & 3 &2\\
3 & 2 & 1 & 3\\
3 & 3 & 2 & 1
\end{array}.\]}
 \item Replacing the symbol $i$ with row $i$ of $D$ to find  a matrix $\mathcal{D}$:
 \small{
\[ \left[ \begin {array}{cccccccccccc} 0&1& -&0&1& -&0&1& -&0&1& -
\\  0&1& -& -&0&1& -&0&1& -&0&1\\  0
&1& -&1& -&0&1& -&0&1& -&0\\   -&0&1&0&1& -& -&0&1&1&
 -&0\\   -&0&1&1& -&0&0&1& -& -&0&1
\\   -&0&1& -&0&1&1& -&0&0&1& -\\  1
& -&0&0&1& -&1& -&0& -&0&1\\  1& -&0& -&0&1&0&1& -&1&
 -&0\\  1& -&0&1& -&0& -&0&1&0&1& -\end {array}
 \right]. \]}
 Applying the restricted Johnson  bound from Theorem \ref{J1}, we can see that the rows of this matrix form  optimal constant weight ternary codes of length $n=12$, weight $w=8$, minimum distance $d=9$ and $A_3(12,9,8)=9$. Applying the second Johnson bound, Theorem \ref{J2}, we note that $A_3(13,9,9)=26$. 
\item Forming a matrix $\mathcal{R}=W \otimes [1 1 1]$:
\small{
\[
\left[ \begin {array}{cccccccccccc} 0&0&0&1&1&1&1&1&1&1&1&1
\\  1&1&1&0&0&0&1&1&1& -& -& -\\  1&
1&1& -& -& -&0&0&0&1&1&1\\  1&1&1&1&1&1& -& -& -&0&0&0
\end {array} \right]. \]}
\item Adding the two matrices $\mathcal{D}$ and $\mathcal{R}$ together with an additional  column of 1's and 0's the matrix
\small{
\[
\arraycolsep=1.0pt\def\arraystretch{1.0}
C=\begin{bmatrix} \mathbf{0} & \mathcal{R }\\ \mathbf{1} & \mathcal{D}  \end{bmatrix}=\\
\left[ \begin {array}{ccccccccccccc} 
0&0&0&0&1&1&1&1&1&1&1&1&1\\  0&1
&1&1&0&0&0&1&1&1& -& -& -\\  0&1&1&1& -& -& -&0&0&0&1
&1&1\\  0&1&1&1&1&1&1& -& -& -&0&0&0\\
1&0&1& -&0&1& -&0&1& -&0&1& -
\\  1&0&1& -& -&0&1& -&0&1& -&0&1
\\  1&0&1& -&1& -&0&1& -&0&1& -&0
\\  1& -&0&1&0&1& -& -&0&1&1& -&0
\\  1& -&0&1&1& -&0&0&1& -& -&0&1
\\  1& -&0&1& -&0&1&1& -&0&0&1& -
\\  1&1& -&0&0&1& -&1& -&0& -&0&1
\\  1&1& -&0& -&0&1&0&1& -&1& -&0
\\  1&1& -&0&1& -&0& -&0&1&0&1& -
  \end {array}
 \right]. \]}
is obtained.
 $C$ is a weighing matrix $W(13,9)$ and the rows of $C$ consist of 13 constant weight ternary codewords of length $n=13$, weight $w=9$ and constant distance $d=9$.
 
\item The final step is to form a $26\times 13$ matrix $$T_{13,9,9}=\begin{bmatrix} C\\-C\end{bmatrix}$$ whose rows form the codewords of  optimal constant weight ternary codes of length $n=13$, $w=9$, and $d=9$, demonstrating that $A_3(13,9,9)=26$. 
\begin{small}

\[
\arraycolsep=1.0pt\def\arraystretch{1.0}
T_{13,9,9}=
 \left[ \begin {array}{ccccccccccccc} 
0&0&0&0&1&1&1&1&1&1&1&1&1
\\  0&1&1&1&0&0&0&1&1&1& -& -& -
\\  0&1&1&1& -& -& -&0&0&0&1&1&1
\\  0&1&1&1&1&1&1& -& -& -&0&0&0
\\1&0&1& -&0&1& -&0&1& -&0&1& -
\\  1&0&1& -& -&0&1& -&0&1& -&0&1
\\  1&0&1& -&1& -&0&1& -&0&1& -&0
\\  1& -&0&1&0&1& -& -&0&1&1& -&0
\\  1& -&0&1&1& -&0&0&1& -& -&0&1
\\  1& -&0&1& -&0&1&1& -&0&0&1& -
\\  1&1& -&0&0&1& -&1& -&0& -&0&1
\\  1&1& -&0& -&0&1&0&1& -&1& -&0
\\  1&1& -&0&1& -&0& -&0&1&0&1& -
\\  0&0&0&0& -& -& -& -& -& -& -& -& -
\\  0& -& -& -&0&0&0& -& -& -&1&1&1
\\  0& -& -& -&1&1&1&0&0&0& -& -& -
\\  0& -& -& -& -& -& -&1&1&1&0&0&0
\\     -&0& -&1&0& -&1&0& -&1&0& -&1
\\   -&0& -&1&1&0& -&1&0& -&1&0& -
\\   -&0& -&1& -&1&0& -&1&0& -&1&0
\\   -&1&0& -&0& -&1&1&0& -& -&1&0
\\   -&1&0& -& -&1&0&0& -&1&1&0& -
\\   -&1&0& -&1&0& -& -&1&0&0& -&1
\\   -& -&1&0&0& -&1& -&1&0&1&0& -
\\   -& -&1&0&1&0& -&0& -&1& -&1&0
\\   -& -&1&0& -&1&0&1&0& -&0& -&1
\end {array}
 \right]. \]
 \end{small}
\end{enumerate}
\end{example}
\begin{remark}
The construction of ternary codes is recursive. The matrix $C$ obtained in the preceding construction is used next to show that $A_3(40,27,27)=80$, etc.
 \end{remark}
 
The first class of optimal ternary codes which will be used in the proof of the second class is introduced next.

\begin{theorem}\label{T1}
For any odd prime power $p$ and positive integer $m$ 
    $$A_3\left(p\left(\frac{p^{m+1}-1}{p-1}\right),p^{m}\left(\frac{p+3}{2}\right),p^{m+1}-1\right)=p^{m+1}.$$
\end{theorem}
\begin{proof}
Let $W$ be the Jacobsthal matrix of order $p$ described in Theorem \ref{Jacob} and note that the distance between any two rows is $\frac{p+3}{2}$. Consider the orthogonal array $\bf O$ in $p$ symbols of Theorem \ref{oa} corresponding to the positive integer $m$. By replacing the $p$ symbols with the rows of $D$ a $p^{m+1}\times p(\frac{p^{m+1}-1}{p-1})$ array, say $T_{p^{m+1}}$, is obtained in which any two distinct rows share exactly $\frac{p^m-1}{p-1}$ rows of $D$ in the same columns.  
Noting this, a careful calculation shows that the distance between any two rows, considered as ternary codes, is $$\bigg(\frac{p^{m+1}-1}{p-1}-\frac{p^{m}-1}{p-1}\bigg)\bigg(\frac{p+3}{2}\bigg)=p^{m}\bigg(\frac{p+3}{2}\bigg).$$ The weight of each code is $w=\big(\frac{p^{m+1}-1}{p-1}\big)(p-1)=p^{m+1}-1$. The rows of $T_{p^{m+1}}$ form the codewords of constant weight ternary code of length $n=p(\frac{p^{m+1}-1}{p-1})$, minimum distance $d=p^{m}(\frac{p+3}{2})$ and  weight $w=p^{m+1}-1$. The condition for the Johnson bound (\ref{eq:1}) is obtained to be $3w^2-4nw+2nd=(\frac{p+3}{p-1})(p^{m+1}-1)$. This and the Johnson bound shows that 
$$A_3\left(p\left(\frac{p^{m+1}-1}{p-1}\right),p^{m}\left(\frac{p+3}{2}\right),p^{m+1}-1\right)\le p^{m+1}.$$
This completes the proof.
\end{proof}

The second class of optimal ternary codes is introduced next.

 \begin{theorem}\label{main}
There is  a weighing matrix $W$ of order $n=\frac{p^{m+1}-1}{p-1}$ and weight $p^m$ for which 
    the rows of $W$ and $-W$ together form the codewords of an optimal constant weight ternary code of length $n$, weight $w=p^m$, and minimum distance $d=p^{m-1}(\frac{p+3}{2})$ for each odd prime power $p$ and integer $m\ge 1$ demonstrating that
    $$A_3\left(\frac{p^{m+1}-1}{p-1},p^{m-1}\left(\frac{p+3}{2}\right),p^{m}\right)=2\left(\frac{p^{m+1}-1}{p-1}\right).$$
\end{theorem}
\begin{proof}
The proof is by induction on $m$. For $m=1$ the statement is to show that $A_3(p+1,\frac{p+3}{2},p)=2(p+1)$. By Theorems \ref{Jacob}, \ref{con-code} and Corollary \ref{cor}, there is a conference matrix $C=W(p+1,p)$ for which $A_3(p+1,\frac{p+3}{2},p)=2(p+1)$. \\ Assuming the existence of a weighing matrix $C_m$ of order $n=\frac{p^{m+1}-1}{p-1}$, weight $w=p^m$ and minimum distance $d=p^{m-1}(\frac{p+3}{2})$ we proceed in three steps.
\begin{enumerate}
\item Considering the $p^{m+1}\times p(\frac{p^{m+1}-1}{p-1})$  matrix $\mathcal{D}$ constructed in Theorem \ref{T1}, the   $(0,\pm 1)$ rows form $p^{m+1}$ codewords of length 
$p(\frac{p^{m+1}-1}{p-1})$, weight $w=p^{m+1}-1$, and minimum distance $p^{m}(\frac{p+3}{2})$.

\item Let $\mathcal{R}=C_m\otimes {\begin{bmatrix} \mathbf{1}_p^t\end{bmatrix}}$. $\mathcal{R}$ is a $\frac{p^{m+1}-1}{p-1}\times p\big(\frac{p^{m+1}-1}{p-1}\big)$ $(0,\pm 1)$-matrix. The weight of each row is $w=p(p^m)=p^{m+1}$ and the minimum distance is $d=p^m(\frac{p+3}{2})$. 
\item The matrices 
$$
W=\begin{bmatrix} \mathbf{0}_{\frac{p^{m+1}-1}{p-1}} & \mathcal{R} \\ \mathbf{1}_{p^{m+1}} & \mathcal{D}  \end{bmatrix},
$$
and $-W$ provides all the codewords. \\
There remains to show that the distance between one codeword from $W $ and one from $-W$ is not smaller than $p^{m-1}(\frac{p+3}{2})$. This is easy to see. The distance of a codeword $k$ in $W$ from a codeword $-j$ in $-W$ is either $p^{m+1}$ or $p^m(\frac{p+3}{2})$ depending on whether $k=j$ or 
$k\ne j$, respectively. $\frac{p+3}{2}\le p$ for all $p\ge 3$. This completes the proof. \qedhere
\end{enumerate}
\end{proof}

\begin{remark} The ternary codes from Theorem \ref{main} corresponding to $p=3$ are in addition \emph{equidistant}. So, there is an optimal set of equidistant constant weight ternary codes consisting of $n=3^{m+1}-1$ codewords of length $n=\frac{3^{m+1}-1}{2}$, weight $w=3^m$ and minimum distance $d=3^m$ for each positive integer $m$. All other ternary codes obtained from Theorem \ref{main} are  2-distance codes.
 \end{remark}
 
 \section{Ternary codes from balanced weighing matrices}
 A weighing matrix $W(v,k)=[w_{ij}]$ is said to be \emph{balanced} if the matrix of absolute values $[|w_{ij}|]$ is the incidence matrix of a symmetric balanced incomplete design with parameters $(v,k,\lambda)$, \cite{ionin} for details. To emphasize that a weighing matrix is balanced it is denoted by $BW(v,k,\lambda)$, where $\lambda=\frac{k(k-1)}{v-1}$.
 Examples of balanced weighing matrices include conference matrices, those with classical parameters $W\big(\frac{p^{m+1}-1}{p-1},p^m\big)$ and a few others. The balanced structure of weighing matrices used in previous sections are instrumental with the generation of optimal codes. A natural question is if all balanced weighing matrices lead to optimal codes. The answer in general depends on the parameters $(v,k,\lambda)$ and it seems to be difficult.
\begin{lemma}\label{d}
Let $C_v$ be the set of ternary codes consisting of the rows of a $BW(v,k,\lambda)=[w_{ij}]$. Then the distance between the codewords is the constant $$d_3=\frac{4k(v-1)-3k(k-1)}{2(v-1)},$$ and the distance between the binary codewords consisting of the rows of
$|W|=[|w_{ij}|]$ is the constant $$d_2=2\left(k-\frac{k(k-1)}{v-1}\right).$$
\end{lemma}
\begin{proof}
The matrix $|W|=[|w_{ij}|]$ is the incidence matrix of a symmetric $(v,k,\lambda)$ design, where $\lambda=\frac{k(k-1)}{v-1}$.  Any two distinct binary codewords  have $\lambda$ ones in the same columns  and the remaining $k-\lambda$ ones have zero in the same columns. This shows that  $d_2=2\left(k-\frac{k(k-1)}{v-1}\right)$. The same arrangement happens for the codewords in $C_v$. Since the two rows are orthogonal, there are $\lambda/2$ (note that this forces $\lambda$ to be even) $-1$'s in the same column contributing the same number to the distance in addition to the $d_2=2\big(k-\frac{k(k-1)}{v-1}\big)$. Therefore, 
\begin{align*}d_3=2\left(k-\frac{k(k-1)}{v-1}\right)+\frac{k(k-1)}{2(v-1)}=\frac{4k(v-1)-3k(k-1)}{2(v-1)}.\end{align*}
This completes the proof. 
\end{proof}

Theorem~\ref{con-code} is extended to the following. 
 The rows of the derived part of a balanced weighing matrix form an optimal constant weight ternary code.
 \begin{theorem}\label{derived}
 Let $\big(v,k,\frac{k(k-1)}{v-1}\big)$ be the parameters of a balanced weighing matrix, and $d=\frac{4k(v-1)-3k(k-1)}{2(v-1)}$. Then
  $$A_3(v-1,d,k-1)=k.$$
 \end{theorem}
 \begin{proof}
 Let $W$ be a balanced weighing matrix $BW(v,k,\lambda) $. 
The rows of the  derived part of $W$ form an optimal constant weight ternary code of length $n=v-1$, constant weight $w=k-1$ and the constant distance $d$. 

 The condition $3w^2-4nw+2nd=4(v-1) - 3(k-1)>0$ by the assumption for the given parameters. By inequality (\ref{eq:1}) of Theorem \ref{J1}
 $$A_3(v-1,d,k-1)\le \frac{4k(v-1)-3k(k-1)}{4(v-1)-3(k-1)}=k.$$
 This completes the proof.
 \end{proof}
A infinite family $BW(1+18\cdot \frac{9^{m+1}-1}{8},9^{m+1},4\cdot9^m)$ including $BW(19,9,4)$ was constructed in \cite{KPS}. As a corollary, we obtain: 
\begin{corollary}\label{cor2}
$A_3(18\cdot \frac{9^{m+1}-1}{8},12\cdot 9^{m},9^{m+1}-1)=9^{m+1}$ for every non-negative integer $m$.
 \end{corollary}
 
 \begin{theorem}\label{last}
 Let $\big(v,k,\frac{k(k-1)}{v-1}\big)$ be the parameters of a balanced weighing matrix for which $v>\frac{3k-1}{2}$. Then
  $$v\le A_3(v,d_3=\frac{4k(v-1)-3k(k-1)}{2(v-1)},k)<2v.$$ 
 \end{theorem}
 \begin{proof}
 Working out the condition in the restricted Johnson bound Theorem \ref{J1} $$3k^2-4vk+2vd_3=3k\bigg(k-\frac{v(k-1)}{v-1}\bigg)$$ which is positive. The upper bound from Theorem \ref{J1} is $\frac{v(4v-3k-1)}{3(v-k)}$.
 It follows that $$v\le A_3(v,d_3,k)<2v$$
 if and only if $v>\frac{3k-1}{2}$.
 \end{proof} 
 \begin{remark}
There is a $BW(19,9,4)$, see \cite{KPS}.  The rows of the derived part of this weighing matrix form an optimal constant weight $8$ and equidistant $d=12$ ternary codes of length $18$  by Theorem \ref{derived}. The condition in Theorem \ref{last} holds and $$19\le A_3(19,12,9)<38.$$
The Johnson upper bound in Theorem \ref{J1} provides a smaller number of $30$ for the possible number of codewords. 
 \end{remark}
 \section*{Acknowledgments.}
The authors acknowledge many help and guidance from Professor Vladimir Tonchev.
Hadi Kharaghani is supported by the Natural Sciences and
Engineering  Research Council of Canada (NSERC).  Sho Suda is supported by JSPS KAKENHI Grant Number 18K03395.

\end{document}